\documentclass[11pt,a4paper]{amsart}
\usepackage[a4paper]{geometry}
\usepackage{amsthm}
\usepackage{amssymb}
\usepackage{amsmath}
\usepackage[utf8]{inputenc}
\usepackage[T1]{fontenc}
\usepackage{comment}
\usepackage{wasysym}
\usepackage{color}

\usepackage{enumitem}
\usepackage{amssymb}
\usepackage{hyperref}
\usepackage{tikz}
\usetikzlibrary{shapes.geometric, arrows}

\tikzstyle{startstop} = [rectangle, rounded corners, minimum width=2cm, minimum height=1cm,text centered,text width=2cm, draw=black]
\tikzstyle{io} = [rectangle, rounded corners, minimum width=2cm, minimum height=1cm,text centered,text width=2cm, draw=black]

\tikzstyle{arrow} = [thick,->,>=stealth]
\tikzstyle{arrow2} = [dashed,->,>=stealth]
\newtheorem{thm}{Theorem}[section]
\newtheorem{cor}[thm]{Corollary}
\newtheorem{lem}[thm]{Lemma}
\newtheorem{prop}[thm]{Proposition}

\theoremstyle{definition}
\newtheorem{defn}[thm]{Definition}
\theoremstyle{remark}
\newtheorem{rem}[thm]{Remark}

\numberwithin{equation}{section}
\DeclareMathOperator{\clco}{\overline{\textnormal{conv}}}

\newcommand{\pten}{\ensuremath{\widehat{\otimes}_\pi}}

\newcommand{\Free}{{\mathcal F}}

\newcommand{\Lip}{{\mathrm{Lip}}_0}

\newcommand{\spa}{\operatorname{span}}
\newcommand{\xast}{x^{\ast}}
\newcommand{\Xast}{X^{\ast}}
\newcommand{\Yast}{Y^{\ast}}

\newcommand{\n}[1]{\left\|#1\right\|}

\title{Stability of diametral diameter TWO properties}
\author{Johann Langemets and Katriin Pirk}

\address{Institute of Mathematics and Statistics, University of Tartu, Narva mnt 18, 51009 Tartu, Estonia}
\email{johann.langemets@ut.ee, katriin.pirk@ut.ee}
\thanks{This work was supported by the Estonian Research Council grants (PSG487) and (PRG877).}
\urladdr{\url{https://johannlangemets.wordpress.com/}}

\keywords {Diameter two property, Daugavet property, $M$-ideal, Köthe--Bochner space, Projective tensor product}
\subjclass[2010]{46B04, 46B20, 46B22, 46B42}
\begin{document}

\begin{abstract}
  We prove that the diametral diameter two properties are inherited by $F$-ideals (e.g., $M$-ideals). On the other hand, these properties are lifted from an $M$-ideal to the superspace under strong geometric assumptions. We also show that all of the diametral diameter two properties are stable under the formation of corresponding Köthe--Bochner spaces (e.g., $L_p$-Bochner spaces). Finally, we investigate when the projective tensor product of two Banach spaces has some diametral diameter two property. 
\end{abstract}

\maketitle

\section{Introduction}

A Banach space $X$ is said to have the \emph{Daugavet property} if every rank-one operator $T\colon X\to X$ satisfies the norm equality
\[
\|I+T\|=1+\|T\|,  
\] 
where $I\colon X\to X$ denotes the identity operator. In 1963, I.~K.~Daugavet (see \cite{Daugavet1963}) was the first to observe that the space $C[0,1]$ has the Daugavet property. Since then this property has attracted a lot of attention and many examples have emerged. Indeed, $C(K)$ spaces for a compact Hausdorff space $K$ without isolated points, $L_1(\mu)$ and $L_\infty(\mu)$ spaces for atomless measure $\mu$, and Lipschitz-free spaces $\mathcal{F}(M)$ and their duals $\text{Lip}_0(M)$ whenever $M$ is a length space -- all have the Daugavet property. 

The following geometric characterisation has been a very useful tool to enlarge the class of spaces with the Daugavet property. Moreover, it has inspired many new properties around the Daugavet property (see e.g., \cite{Abrahamsen2013}, \cite{Becerra2018a}, \cite{Haller2019}, \cite{Haller2020}).

\begin{lem}[see {\cite[Lemma~2.1]{Kadets2000}}]\label{lem: DP slices}
  \label{thm:kadets Daugavet}
  A Banach space $X$ has the Daugavet property if and only if for every $x\in S_X$, every slice $S$ of $B_X$, and every $\varepsilon > 0,$ there exists $y \in S$ such that $\|x - y\| \ge 2 - \varepsilon.$ 
 \end{lem}

 Observe that Lemma~\ref{lem: DP slices} immediately gives that in a space with the Daugavet property every slice of the unit ball has diameter two. Moreover, it is known that spaces with the Daugavet property satisfy that all nonempty relatively weakly open subsets of the unit ball (e.g., slices) have diameter two (see \cite{Shvydkoy2000}). A Banach space with such a property is said to have the \emph{diameter two property} (D2P in short). If a Banach space satisfies that all of the slices of the unit ball have diameter two, then it is said that the space has the \emph{local diameter two property} (LD2P in short). Clearly, the Daugavet property implies the D2P and the D2P implies the LD2P. However, both of the converses are not true in general. The space $c_0$ is an example of a space with the D2P and failing the Daugavet property. The existence of a space with the LD2P and without the D2P was first proven in \cite{Becerra2015}, another example can be found in \cite{Abrahamsen2020}.

In \cite{Ivakhno2004} another weakening of the Daugavet property was considered: a Banach space $X$ is said to be a \emph{space with bad projections} (SBP in short) if for every rank-one projection $P\colon X\to X$ one has that $\|I-P\|\geq 2$. Y.~Ivakhno and V.~Kadets obtained the following geometrical characterization of a space with bad projections.

\begin{thm}[see {\cite[Theorem 1.4]{Ivakhno2004}}]
 \label{thm:ik-sbp-dld2p}
 A Banach space $X$ is SBP if and only if for every slice $S$ of $B_X$, every unit sphere element $x \in S$, and every $\varepsilon > 0,$ there exists $y \in S$ such that $\|x - y\| \ge 2 - \varepsilon.$ 
\end{thm}

It is clear that an SBP space has the LD2P in a strong sense since for any slice $S$ of the unit ball and any unit sphere element $x$ in $S$ there exists a point $y$ in $S$ almost diametral to $x$. Inspired by this the authors of \cite{Becerra2018a} called the equivalent formulation of an SBP space appearing in Theorem~\ref{thm:ik-sbp-dld2p} the \emph{diametral local diameter two property} (DLD2P in short). Note that the DLD2P has also appeared under the name of LD2P+ in \cite{Abrahamsen2016a}. From now on we will use the name DLD2P.

In \cite{Becerra2018a} also the diametral analogue of the D2P was introduced and studied: a Banach space $X$ is said to have the \emph{diametral diameter two property} (DD2P in short) if for every nonempty relatively weakly open subset $U$ of $B_X$, every unit sphere element $x\in U$, and every $\varepsilon>0$ there exists $y\in U$ such that $\|x-y\|\geq 2-\varepsilon$. From \cite[Lemma~3]{Shvydkoy2000} one has that the Daugavet property implies even the DD2P. An example of a Banach space with the DD2P and failing the Daugavet property was given in \cite[Example 2.2]{Becerra2018a}. The questions whether there is a space with the DLD2P and without the DD2P is still open (see \cite[Question 4.1]{Becerra2018a}).

We need also the very recent definition of a $\Delta$-point emerging from \cite{Abrahamsen2020a} before we can start to describe the contents of this paper.
In the language of slices, a $\Delta$-point is a unit sphere element $x$ such that every slice $S$ of the unit ball that contains $x$ also contains a point $y$ almost diametral to $x$.

Observe that a Banach space $X$ has the DLD2P if and only if every point in the unit sphere of $X$ is a $\Delta$-point (see Lemma \ref{lem: Delta point crit}). A Banach space $X$ has the \emph{convex diametral local diameter two property} (convex DLD2P in short) if its unit ball is the closed convex hull of its $\Delta$-points. It is clear that the DLD2P implies the convex DLD2P, however $\ell_\infty$ is an example of a space with the convex DLD2P and failing the DLD2P (see \cite[Corollary~5.4]{Abrahamsen2020a}). That the convex DLD2P implies the LD2P was shown in \cite[Proposition~5.2]{Abrahamsen2020a}. Finally, note that $c_0$ has the LD2P, but fails the convex DLD2P \cite[Remark~5.5]{Abrahamsen2020a}.

We summarize all of the aforementioned properties and their relations in the following diagram. 
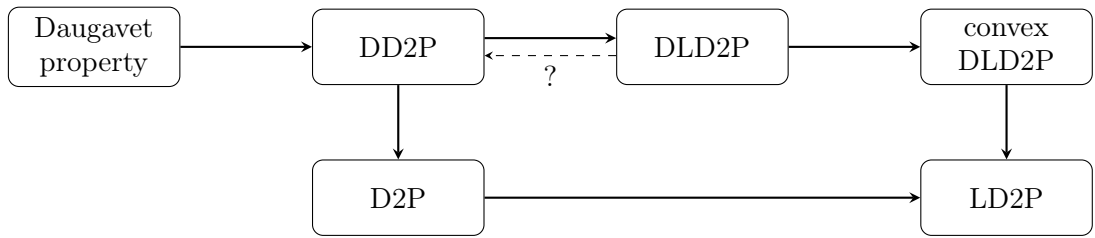
\begin{center}

\begin{tikzpicture}[node distance=2cm]
  \node (Daugavet) [startstop] {Daugavet property};
  \node (DD2P) [io, right of=Daugavet, xshift=2cm] {DD2P};
  \node (DLD2P) [io, right of=DD2P, xshift=2cm] {DLD2P};
  \node (convexDLD2P) [io, right of=DLD2P, xshift=2cm] {convex DLD2P};
  \node (D2P) [io, below of=DD2P] {D2P};
  \node (LD2P) [io, below of=convexDLD2P] {LD2P};

  \draw [arrow] (Daugavet) --  (DD2P);
  \draw [arrow,transform canvas={yshift=0.7ex}] (DD2P) --  (DLD2P);
  \draw [arrow2,transform canvas={yshift=-0.7ex}] (DLD2P) -- node[anchor=north,] {?} (DD2P);
  \draw [arrow] (DLD2P) --  (convexDLD2P);
  \draw [arrow] (DD2P) --  (D2P);
  \draw [arrow] (D2P) -- (LD2P);
  \draw [arrow] (convexDLD2P) --  (LD2P);
  \end{tikzpicture}
\end{center}

The main aim of this paper is to study stability properties of diametral diameter two properties in the context of $F$-ideals, Köthe--Bochner spaces, and projective tensor products of Banach spaces. 

It is known that the Daugavet property passes down from the superspace to its $M$-ideals (see \cite{Kadets2000}). The Daugavet property lifts from an $M$-ideal to the superspace if the quotient space with respect to the given $M$-ideal also has the Daugavet property (see \cite{Kadets2000}). On the other hand, it is known that the LD2P and the D2P always lift from an $M$-ideal to the superspace (see \cite{Haller2014a}). In Section~\ref{sec: F-ideals}, we prove that the DD2P and the DLD2P pass down to a wide class of $F$-ideals, thus also to $M$-ideals (see Theorem~\ref{thm: DD2P F-ideals}) and that these properties lift from an $M$-ideal to the superspace whenever the quotient space also has the corresponding diametral diameter two property (see Proposition~\ref{prop: DD2P lifts}). However, the convex DLD2P behaves similarly to the regular diameter two property (see Proposition~\ref{prop: convex and M-ideal}).

Recall that if $X$ has the D2P, then so does $L_p(\mu,X)$ for every $1\leq p\leq \infty$ (see \cite{Acosta2015} and \cite{Becerra2006}). This result was recently extended by J.-D.~Hardtke who proved that all of the regular diameter two properties are even stable by forming the corresponding Köthe--Bochner space (see \cite{Hardtke2020}). For the Daugavet property it is known that $L_1(\mu,X)$ and $L_\infty(\mu, X)$ have the Daugavet property whenever $X$ has it or $\mu$ is atomless. However, if $1<p<\infty$, then $L_p(\mu,X)$ can never have the Daugavet property. We will show that all of the diametral diameter two properties are also stable by forming the corresponding Köthe--Bochner space (see Theorem~\ref{thm: Köthe-Bochner}), thus also for $L_p$-Bochner spaces (see Corollary~\ref{cor: DD2P in L_p}).

The question whether $X\pten Y$ has the Daugavet property whenever both $X$ and $Y$ have the Daugavet property (see \cite{Werner2001}) is open to this day. Very recently M.~Mart\'in and A.~Rueda Zoca made good progress towards this question by introducing a formally stronger version of the Daugavet property called the \textit{weak operator Daugavet property} (WODP in short) (see Definition~\ref{def: WODP}). They were able to prove that $X\pten Y$ has the Daugavet property whenever both $X$ and $Y$ have the WODP (see \cite{Martin2020}). We will prove that for the DLD2P of $X\pten Y$ it suffices to assume that $X$ has the WODP (see Proposition~\ref{prop: DLD2P of proj ten}). For the convex DLD2P of $X\pten Y$ it suffices to assume that $X$ has the convex DLD2P (see Proposition~\ref{prop: convex DLD2P and proj tensor product}). As a consequence we get that the Lipschitz-free space  $\mathcal{F}(M)$ has the convex DLD2P if and only if  $\mathcal{F}(M,X)$ has the convex DLD2P for any Banach space $X$ (see Corollary~\ref{cor: convexDLD2P in F(M)}).

 \section{Notation and preliminary results}\label{sec: preliminary}
 All Banach spaces considered in this paper are nontrivial and over
the real field. The closed unit ball of a Banach space $X$ is denoted
by $B_X$ and its unit sphere by $S_X$. The dual space of $X$ is
denoted by $X^\ast$ and the bidual by $X^{\ast\ast}$.
By a \emph{slice} of $B_X$ we mean a set of the form
\begin{equation*}
  S(B_X, x^*,\alpha) :=
  \{
  x \in B_X : x^*(x) > 1 - \alpha
  \},
\end{equation*}
where $x^* \in S_{X^*}$ and $\alpha > 0$. Notice that nonempty finite intersections of slices of the unit ball form a basis for the inherited weak topology of $B_X$. Therefore, in the definitions of the D2P and the DD2P, it suffices to consider nonempty finite intersections of slices instead of nonempty relatively weakly open subsets.
 
For $x\in S_X$ and $\varepsilon>0$ we denote
\[
\Delta_{\varepsilon}^X(x):=\{y\in B_X\colon \|x-y\|\geq 2-\varepsilon \}  
\]
and 
\[
\Delta_X =\{x\in S_X\colon x\in \overline{\text{conv}}\Delta_{\varepsilon}^X(x) \quad \text{for all $\varepsilon>0$} \}.  
\]

Recall from \cite{Abrahamsen2020a} that that an element $x\in S_X$ is said to be a \emph{$\Delta$-point} if $x\in \overline{\text{conv}} \Delta_{\varepsilon}^X(x)$ for every $\varepsilon>0$. Note that $\Delta$-points can be also characterized in terms of slices.

\begin{lem}[see {\cite[Lemma 2.1]{Abrahamsen2020a}}]
\label{lem: Delta point crit}
  Let $X$ be a Banach space and $x \in S_X$.
  The following assertions are equivalent:
  \begin{enumerate}
  \item[$(i)$]\label{aaa}
    $x$ is a $\Delta$-point;
  \item[$(ii)$]\label{bbb}
    for every slice $S$ of $B_X$ with $x\in S$ and every
    $\varepsilon > 0$ there exists $y \in S$
    such that $\|x - y\| \geq 2 - \varepsilon$.
  \end{enumerate}
\end{lem}

Thus, by Lemma~\ref{lem: Delta point crit}, it is clear that a Banach space $X$ has the DLD2P if and only if $S_X=\Delta_X$. On the other hand, since the convex DLD2P of $X$ is defined by the fact that $B_X=\overline{\text{conv}} \Delta_X$, then by a Hahn--Banach separation argument one has the following.

\begin{lem}\label{lem: convex DLD2P crit}
  Let $X$ be a Banach space. Then $X$ has the convex DLD2P if and only if $S\cap \Delta_X\neq \emptyset$ for every slice $S$ of $B_X$.
\end{lem}

We recall that a norm $F$ on $\mathbb R^2$ is called \emph{absolute} (see \cite{Bonsall1973}) if
\[
F(a,b)=F(|a|,|b|)\qquad\text{for all $(a,b)\in\mathbb R^2$}
\]
and \emph{normalized} if
\[
F(1,0)=F(0,1)=1.
\]
For example, the $\ell_p$-norm $\|\cdot\|_p$ is absolute and normalized for every $p\in[1,\infty]$. If $F$ is an absolute normalized norm on $\mathbb R^2$ (see \cite[Lemmata~21.1 and 21.2]{Bonsall1973}), then 

\begin{itemize}
\item $\|(a,b)\|_\infty\leq F(a,b)\leq \|(a,b)\|_1$ for all $(a,b)\in\mathbb R^2$;
\item if $(a,b),(c,d)\in\mathbb R^2$ with $|a|\leq |c|\quad \text{and}\quad |b|\leq |d|,$ then \[F(a,b)\leq F(c,d);\]
\item the dual norm $F^\ast$ on $\mathbb R^2$ defined by 
\[
F^\ast(c,d)=\max_{F(a,b)\leq 1}(|ac|+|bd|) \qquad\text{for all $(c,d)\in\mathbb R^2$}
\]
is also absolute and normalized. Note that $(F^\ast)^\ast=F$.
\end{itemize}

If $X$ and $Y$ are Banach spaces and $F$ is an absolute normalized norm on $\mathbb R^2$, then we denote by $X\oplus_F Y$ the product space $X\times Y$ with respect to the norm
\[
\|(x,y)\|_F=F(\|x\|,\|y\|) \qquad\text{for all $x\in X$ and $y\in Y$}.
\]
In the special case where $F$ is the $\ell_p$-norm, we write $X\oplus_p Y$.
Note that $(X\oplus_F Y)^\ast=X^\ast\oplus_{F^\ast} Y^\ast$.

The following lemma is easily verified from the definitions.

\begin{lem}\label{lemma: extreme point is an strongly exposed point}
	Let $F$ be an absolute normalized norm on $\mathbb{R}^{2}$ such that $(1,0)$ is an extreme point of the unit ball $B_{(\mathbb{R}^2, F)}$. Then $(1,0)$ is a strongly exposed point of $B_{(\mathbb{R}^2, F)}$, which is strongly exposed by the functional $(1,0)\in B_{(\mathbb{R}^2, F^\ast)} $. In particular, for every $\varepsilon>0$ there is a $\gamma>0$ such that, whenever $(a,b)\in B_{(\mathbb{R}^2, F)}$ and $a>1-\gamma$, then $|b|<\varepsilon$.
\end{lem}

Finally, given two Banach spaces $X$ and $Y$, we will denote by $X\pten Y$ the projective tensor product of $X$ and $Y$. Recall that the space $\mathcal{L}(Y,X^\ast)$ of bounded linear operators from $Y$ to $X^\ast$ is linearly isometric to the topological dual of $X\pten Y$. We refer to \cite{Ryan2002} for a detailed treatment and applications of tensor products.

\section{$F$-ideals}\label{sec: F-ideals}

We denote the \emph{annihilator} of a subspace $Y$ of a Banach space $X$ by
\[
Y^{\perp}=\{\xast\in \Xast \colon \xast(y)=0\quad \text{for all } y\in Y\}.
\]

According to the terminology in \cite{Harmand1993}, a closed subspace $Y\subset X$ is called an \emph{M-ideal} if there exists a projection $P$ on $\Xast$ with $\ker P=Y^{\perp}$ and
\[
\n{\xast}=\n{P\xast}+\n{\xast-P\xast}\qquad \text{for all } \xast\in \Xast.
\]
Thus, if $Y$ is an $M$-ideal in $X$, then $\Xast$ decomposes as $\Xast=\Yast\oplus_1 Y^\perp$.

In \cite{Kadets2000}, the heritance of the Daugavet property with respect to $M$-ideals was studied and the authors obtained the following result. 
\begin{thm}[see {\cite[Propositions~2.10 and 2.11]{Kadets2000}}]\label{thm: Daugavet M-ideal}
  Let $X$ be a Banach space and $Y$ an $M$-ideal in $X$. 
  \begin{itemize}
      \item[$(a)$] If $X$ has the Daugavet property, then $Y$ has the Daugavet property;
      \item[$(b)$] If $Y$ and $X/Y$ share the Daugavet property, then $X$ has the Daugavet property.
  \end{itemize}
\end{thm}

Therefore, the Daugavet property passes down from the superspace to its $M$-ideals. However, the converse is false in general. Indeed, $C[0,1]$ is an $M$-ideal in $C[0,1]\oplus_\infty \ell_2$, but $C[0,1]\oplus_\infty \ell_2$ even fails the DLD2P (see \cite[Theorem~3.2]{Ivakhno2004}).

On the other hand, it is known that the regular diameter two properties always lift from an $M$-ideal to the superspace (see \cite{Haller2014a}). Our aim in this section is to show that the diametral diameter two properties behave similarly to the Daugavet property, while the convex DLD2P resembles the regular diameter two properties with respect to $M$-ideals.

We will now consider more general types of ideals (see \cite[pp.~45--46]{Harmand1993}). Let $F$ be an absolute and normalized norm on $\mathbb R^2$. A closed subspace $Y\subset X$ is called an \emph{F-ideal} if there exists a projection $P$ on $\Xast$ with $\ker P=Y^{\perp}$ and
\[
\n{\xast}=F^\ast(\n{P\xast},\n{\xast-P\xast})\qquad \text{for all } \xast\in \Xast.
\] 

Therefore, $M$-ideals are just the $\|\cdot\|_\infty$-ideals. We are now ready to prove our main result in this section.

\begin{thm}\label{thm: DD2P F-ideals}
  Let $X$ be a Banach space, $F$ be an absolute normalized norm on $\mathbb R^2$ such that $(1,0)$ is an extreme point of $B_{(\mathbb R^2,F^\ast)}$, and $Y$ an $F$-ideal in $X$. 
  \begin{itemize}
      \item[$(a)$] If $X$ has the DD2P, then $Y$ has the DD2P;
      \item[$(b)$] If $X$ has the DLD2P, then $Y$ has the DLD2P.
    \end{itemize}
  \end{thm}
\begin{proof}
    $(a)$ By our assumptions $X^\ast=Y^\ast \oplus_{F^\ast} Y^\perp$. Let $P\colon \Xast \to \Xast$ with $\ker P = Y^\perp$ be the $F$-ideal projection. Let $\varepsilon>0$,  $y_i^\ast \in S_{Y^\ast}$, and $y\in \bigcap_{i=1}^n S(B_Y, y_i^\ast, \varepsilon)$ with $\|y\|=1$. Our aim is to show that there exists $u\in \bigcap_{i=1}^n S(B_Y, y_i^\ast, \varepsilon)$ such that $\|y-u\|\geq 2-\varepsilon$. 
    
    Pick $\delta>0$ such that $y_i^\ast(y)>1-\varepsilon+\delta$ for every $i\in \{1,\dots,n\}$. Consider the slices $S(B_X, Px_i^\ast, \varepsilon-\delta)$, where $x_i^\ast$ are norm preserving extensions of $y_i^\ast$. 
    
    Since $X$ has the DD2P and $y\in \bigcap_{i=1}^n S(B_X, Px_i^\ast, \varepsilon-\delta)$ we can find an element $w\in \bigcap_{i=1}^n S(B_X, Px_i^\ast, \varepsilon-\delta)$ such that $\|y-w\|\geq 2-\varepsilon/3$. Find $z^\ast\in \Xast$ with $\|z^\ast\|=1$ such that $z^\ast(y-w)\geq 2-\varepsilon/3$, which gives that $z^\ast(y)=Pz^\ast(y)\geq 1-\varepsilon/3$. Therefore, by Lemma~\ref{lemma: extreme point is an strongly exposed point},
    \[
    \|Pz^\ast\|\geq 1-\frac{\varepsilon}3\quad \text{ and }\quad \|z^*-Pz^*\|\leq \frac{\varepsilon}3.
    \]

Now, similarly to \cite[Remark I.1.13]{Harmand1993}, one can show that $B_Y$ is $\sigma(X,Y^*)$-dense in $B_X$. Thus, we can find a $u\in B_Y$ such that for every $i\in \{1,\dots,n\}$
$$|Px_i^*(w - u)| < \delta\quad \textnormal{and}\quad |Pz^*(w-u)|<\frac{\varepsilon}{3}.$$
Then $u\in \bigcap_{i=1}^n S(B_Y,y_i^*,\varepsilon),$ because for every $i\in \{1,\dots,n\}$ we have
\begin{align*}
    y_i^*(u) &= (Px_i^*)(u)\\&
    = (Px_i^*)(w) - (Px_i^*)(w-u)\\&
    > 1 - (\varepsilon-\delta) -\delta\\&
    = 1-\varepsilon.
\end{align*}
Notice that
\begin{align*}
    \|y-u\| &\geq z^*(y-u)\\&
    = z^*(y - w) + Pz^*(w -u) + (z^* - Pz^*)(w)\\&
    \geq 2-\frac{\varepsilon}{3} - \frac{\varepsilon}{3} - \frac{\varepsilon}{3}\\&
    = 2-\varepsilon.
\end{align*}
Therefore, $Y$ has the DD2P.

  $(b)$ Take $n=1$ in the proof of $(a)$.
\end{proof}

As an application of Theorem~\ref{thm: DD2P F-ideals}, we obtain the known stability results for the diametral diameter two properties in $M$-ideals (see \cite[Propositions~2.3.4 and 2.3.5]{Pirk2020}).

\begin{cor}
  Let $X$ be a Banach space and $Y$ an $M$-ideal in $X$. 
  \begin{itemize}
    \item[$(a)$] If $X$ has the DD2P, then $Y$ has the DD2P;
    \item[$(b)$] If $X$ has the DLD2P, then $Y$ has the DLD2P.
  \end{itemize}
\end{cor}

In \cite[Problem~8]{Pirk2020} it was asked whether a similar statement as Theorem~\ref{thm: Daugavet M-ideal}, (b), also holds for the diametral diameter two properties. In the following Proposition we will give a positive answer to this question.

\begin{prop}\label{prop: DD2P lifts}
  Let $X$ be a Banach space and $Y$ an $M$-ideal in $X$. 
  \begin{itemize}
      \item[$(a)$] If $Y$ and $X/Y$ share the DD2P, then $X$ has the DD2P;
      \item[$(b)$] If $Y$ and $X/Y$ share the DLD2P, then $X$ has the DLD2P.
  \end{itemize}
  \end{prop}
  
  We omit the proof of Proposition~\ref{prop: DD2P lifts}, because it
is very similar to the proof of \cite[Proposition 2.11]{Kadets2000}.

 We end this section by studying how the convex DLD2P behaves in the context of $M$-ideals. Recall from \cite{Abrahamsen2020a} that the convex DLD2P is not inherited by $M$-ideals in general, since $c_0$ is an $M$-ideal in $\ell_\infty$, but even though the latter has the convex DLD2P, $c_0$ does not. On the other hand, we will now show that the convex DLD2P carries over to the whole space from its proper $M$-ideal similarly to the the regular diameter two properties \cite{Haller2014a}.

\begin{prop}\label{prop: convex and M-ideal}
  Let $X$ be a Banach space and $Y$ an $M$-ideal in $X$. If $Y$ has the convex DLD2P, then $X$ has the convex DLD2P.
\end{prop}
\begin{proof}
The proof is modelled on the proof of \cite[Proposition~3]{Haller2014a}. Let $S(B_{X},x^*,\alpha)$ be a slice and let $\varepsilon > 0$. Let $P:X^* \to X^*$ with $\ker P = Y^\perp$ be the $M$-ideal projection.
  Define
  \begin{equation*}
    y^* := \frac{Px^*}{\|Px^*\|}
    \quad \mbox{and} \quad
    \beta :=
    \frac{\varepsilon(1-\|Px^*\|)+\varepsilon^2}{\|Px^*\|} > 0.
  \end{equation*}
  If $Px^*=0$, we can take $y^*\in S_{Y^*}$ arbitrary and set $\beta:= \alpha$. Since $Y$ has the convex DLD2P there exist
  $y \in S(B_Y,y^*,\beta)\cap \Delta_Y$. Thus we can find $n\in \mathbb N$ and $y_1,\dots, y_n\in B_Y$ such that
  \[
  \|y-1/n\sum_{i=1}^n y_i\|\leq \varepsilon\quad \text{ and } \quad\|y_i-y\|\geq 2-\varepsilon \text{ for all $i\in\{1,\dots,n\}$.}
  \]
       The choice of $\beta$ means that
  \begin{equation*}
    Px^*(y) > (\|Px^*\| - \varepsilon)(1+\varepsilon).
  \end{equation*}
  Find $x\in B_X$ such that 
  \begin{equation*}
    (x^*-Px^*)(x) > 
    (\|x^*-Px^*\| - \varepsilon)(1 + \varepsilon).
  \end{equation*}
    By Proposition~2.3 in \cite{Werner1994}, we may choose $z \in Y$ such that
    $$\|y_i + x - z\| < 1+\varepsilon,\quad
    \|y + x - z\| < 1+\varepsilon,\quad\text{and}\quad
    |Px^*(x-z)| < \varepsilon.$$
  
  Define
  \begin{equation*}
    u_i := \frac{y_i + x - z}{1+\varepsilon}
    \quad \mbox{and} \quad
    u := \frac{y+x-z}{1+\varepsilon}.
  \end{equation*}
  Then
  \begin{align*}
    x^*(u) &
    = \frac{x^*(y + x - z)}{1+\varepsilon} \\
    &= \frac{Px^*(y) + (x^*-Px^*)(x) + Px^*(x-z)}
    {1+\varepsilon} \\
    &> \frac{(\|Px^*\|-\varepsilon)(1+\varepsilon) +
      (\|x^*-Px^*\|-\varepsilon)(1+\varepsilon) - \varepsilon}
    {1+\varepsilon} \\
    &> \|x^*\| - 3\varepsilon=1-3\varepsilon. 
  \end{align*}
  Also 
  $\|u-1/n\sum_{i=1}^n u_i\|\leq \varepsilon(1+\varepsilon)$ and $\|u_i-u\|(1+\varepsilon)=\|y_i-y\|\geq 2-\varepsilon$ for all $i\in\{1,\dots,n\}$.
  
  Note that we can assure $\|u\|=1$ by scaling argument. Moreover, since $\varepsilon > 0$ is arbitrary
  we can choose it as small as we like so that
  $u \in S(B_{X},x^*,\alpha)\cap \Delta_X$,
 that is, $X$ has the convex DLD2P.
\end{proof}

\begin{rem}
    Note that one cannot generalize Proposition~\ref{prop: convex and M-ideal} to $F$-ideals, because $Y=\ell_\infty$ has the convex DLD2P and is an $F$-ideal in $X=\ell_\infty\oplus_2 \mathbb{R}$, however $X$ even fails the LD2P.
\end{rem}

\section{Köthe-Bochner spaces}\label{sec: Köthe--Bochner}

We begin by repeating the presentation of Köthe--Bochner spaces as given in \cite{Hardtke2020}. Let $(S,\mathcal{A}, \mu)$ be a complete, $\sigma$-finite measure space. A \emph{Köthe function space} over $(S,\mathcal{A}, \mu)$ is a Banach space $(E, \|\cdot\|_E)$ of real-valued measurable functions on $S$ modulo equality $\mu$-a.e. such that
\begin{itemize}
  \item[(1)] $\chi_A\in E$ for every $A\in \mathcal{A}$ with $\mu(A)<\infty$;
  \item[(2)] for every $f\in E$ and every $A\in\mathcal{A}$ with $\mu(A)<\infty$ one has that $f$ is $\mu$-integrable over $A$;
  \item[(3)] if $g$ is measurable and $f\in E$ such that $|g(t)|\leq |f(t)|$ $\mu$-a.e., then $g\in E$ and $\|g\|_E\leq \|f\|_E$.  
\end{itemize}

If $X$ is a Banach space, then a function $f\colon S\to X$ is called \emph{simple} if there are finitely many disjoint sets $A_1,\dots, A_n\in \mathcal{A}$ such that $\mu(A_i)<\infty$ for every $i\in \{1,\dots,n\}$, $f$ is constant on each $A_i$, and $f(t)=0$ if $t\in S\setminus(\bigcup_{i=1}^n A_i)$. The function $f$ is said to be \emph{Bochner-measurable} if there exists a sequence $(f_n)$ of simple functions such that $\lim_{n\to \infty} \|f_n(t)-f(t)\|=0$ $\mu$-a.e.

For a Köthe function space $E$ and a Banach space $X$ we denote by $E(X)$ the space of all Bochner-measurable functions $f\colon S\to X$ such that $\|f(\cdot)\|\in E$. If we equip $E(X)$ with the norm $\|f\|_{E(X)}=\| \|f(\cdot)\| \|_E$, then $E(X)$ is a Banach space and it is called the \emph{Köthe--Bochner space}. The classical examples of Köthe--Bochner function spaces are the Lebesgue--Bochner spaces $L_p(\mu, X)$ for $1\leq p\leq \infty$. We refer the reader to \cite{Lin2004} for more information on Köthe-Bochner spaces.

J.-D.~Hardtke proved in \cite{Hardtke2020} that the regular diameter two properties are stable by forming Köthe--Bochner spaces. He used the following general technique: if a Banach space property can be described via test families (see \cite[Definition~1.1]{Hardtke2020}) and the property is stable by finite absolute sums, then it is also stable under the formation of corresponding Köthe--Bochner spaces. Since the DLD2P and the convex DLD2P can be solely characterized by $\Delta$-points (see Section~\ref{sec: preliminary}), then it is not hard to describe them via test families. By \cite[Theorem~3.2]{Ivakhno2004} (resp. \cite[Theorem~5.8]{Abrahamsen2020a}), the DLD2P (resp. convex DLD2P) is stable by finite absolute sums, hence Hardtke's technique yields that both of these properties are stable by forming Köthe--Bochner spaces also. However, it is not clear whether the DD2P can be characterized in terms of test families. For this reason and for the simplicity of the direct proofs for the DLD2P and the convex DLD2P we will provide them here.


\begin{thm}\label{thm: Köthe-Bochner}
    Let $(S,\mathcal{A}, \mu)$ be a complete, $\sigma$-finite measure space, $E$ a Köthe function space over $(S,\mathcal{A}, \mu)$ and $X$ a Banach space such that the simple functions are dense in $E(X)$.
    \begin{itemize}
        \item[(a)] If $X$ has the DD2P, then $E(X)$ also has the DD2P;
        \item[(b)] If $X$ has the DLD2P, then $E(X)$ also has the DLD2P;
        \item[(c)] If $X$ has the convex DLD2P, then $E(X)$ also has the convex DLD2P.
    \end{itemize}
\end{thm}
\begin{proof}
$(a)$
	Let $f\in \bigcap_{i=1}^n S(B_{E(X)},f_i^\ast, \alpha_i)$ with $\|f\|_{E(X)}=1$, where $f_i^\ast\in (E(X))^\ast$ are with norm 1, and $\varepsilon>0$. We need to find a $g\in \bigcap_{i=1}^n S(B_{E(X)},f_i^\ast,\alpha_i)$ such that $\|f-g\|\geq 2-\varepsilon$. Since the simple functions are dense in $E(X)$, we may assume that $f=\sum_{k=1}^{m} x_k \chi_{A_k}$, where $A_k$ are pairwise disjoint measurable sets of $\mathcal{A}$  and $x_k\in X$.
	
	Define $x^\ast_{ik}$ by $x^\ast_{ik}(x)=f_i^\ast(x\cdot \chi_{A_k})$. Note that $x^\ast_{ik}\in X^\ast$ for every $i\in \{1,\dots, n\}$. Clearly, $x^\ast_{ik}$ is linear and it is even continuous. Indeed, if $x\in B_X$, then for every $i\in \{1,\dots,n\}$ we have
	\begin{align*}
	|x^\ast_{ik}(x)|&=|f_i^\ast(x\cdot \chi_{A_k})|\leq \|f_i^\ast\|\cdot\|x\cdot \chi_{A_k}\|\\
	&= \|f_i^\ast\|\cdot\|x\|\cdot \|\chi_{A_k}\|\leq \|f_i^\ast\|\cdot\|\chi_{A_k}\|.
	\end{align*}
	Thus $\|x^\ast_{ik}\|\leq \|f_i^\ast\|\|\chi_{A_k}\|$.
	
	Now we want to construct $g=\sum_{k=1}^{m} y_k \chi_{A_k}$. If $x_k=0$, we take $y_k=0$. If $x_k\neq 0$, then consider the slices $$
	S_{ik}=\{y\in \|x_k\|B_X\colon x^\ast_{ik}(y)>x^\ast_{ik}(x_k)-\alpha_{ik}\},$$
	where $\alpha_{ik}>0$ are such that $\sum_{k=1}^m\alpha_{ik}\leq f_i^\ast(f)-(1-\alpha_i)$.
	Since $X$ has the DD2P, there are $y_k\in \bigcap_{i=1}^n S_{ik}$ such that $\|x_k-y_k\|>(2-\varepsilon)\|x_k\|$ for every $k\in\{1,\dots,m\}$ for which $x_k\neq 0$. 
	
	From $\|y_k\|\leq \|x_k\|$, we conclude that $\|g(\cdot)\|\leq \|f(\cdot)\|$. This implies that $g\in E(X)$ and $\|g\|\leq 1$.
	
	Observe that $g\in \bigcap_{i=1}^n S(B_{E(X)},f_i^\ast,\alpha_i)$. Indeed, we have for every $i\in \{1,\dots,n\}$ that
	\begin{align*}
	f_i^\ast(g)&=f_i^\ast\Big(\sum_{k=1}^m y_k\chi_{A_k}\Big)=\sum_{k=1}^m x^\ast_{ik}(y_k)> \sum_{k=1}^m x^\ast_{ik}(x_k)-\sum_{k=1}^m \alpha_{ik}\\
	&=f_i^\ast(f)-\sum_{k=1}^m \alpha_{ik}\geq 1-\alpha_i.
	\end{align*}

	Finally, if $s\in A_k$, then $\|f(s)-g(s)\|=\|x_k-y_k\|>(2-\varepsilon)\|x_k\|$. Therefore, 
	\[
	\|f-g\|>(2-\varepsilon)\|f\|=2-\varepsilon,
	\]
	and we are done.
	
	$(b)$
	Take $n=1$ in the proof of $(a).$
	
	$(c)$
		Let $f^*\in S_{{E(X)}^*}$ and $\alpha, \varepsilon>0$. Our aim is to show that there exists $f\in S(B_{E(X)}, f^*,\varepsilon)\cap S_{E(X)}$ such that $f$ is a $\Delta$-point in $X$, i.e. $f\in \clco\Delta_\varepsilon(f)$ for every $\varepsilon>0.$
	
	Fix an arbitrary $f_0\in S(B_{E(X)}, f^*,\alpha)$. 
	Since the simple functions are dense in $E(X)$, we may assume that $f_0=\sum_{k=1}^{m} x_k \chi_{A_k}$, where $A_k$ are pairwise disjoint measurable sets of $\mathcal{A}$  and $x_k\in X$.
	
	Define $x^\ast_{k}$ by $x^\ast_{k}(x)=f^\ast(x\cdot \chi_{A_k})$. Note that, analogically to part $(a)$ it can be shown that $x^\ast_{k}\in X^\ast$.
	
	Now we want to construct $f=\sum_{k=1}^{m} y_k \chi_{A_k}$. If $x_k=0$, we take $y_k=0$. If $x_k\neq 0$, then consider the slices
	$$S_k=\{y\in \|x_k\|B_X\colon x^\ast_k(y)>x^\ast_k(x_k)-\alpha_k\},$$
	where $\alpha_k>0$ are such that $\sum_{k=1}^m\alpha_k\leq f^\ast(f)-(1-\alpha)$.
	Let $\delta> 0$.

	Since $X$ has the convex DLD2P, there are $y_k\in S_k$ such that $y_k/\|x_k\|$ is a $\Delta$-point for every $k\in\{1,\dots,m\}$, hence, there exist $y_{ki}\in \|x_k\|B_X$ such that
	$$\Big\Vert y_k - \sum_{i=1}^n \frac{1}{n} y_{ki}\Big\Vert < \delta \|x_k\|$$
	and $\|y_k - y_{ki}\| \geq (2-\varepsilon)\|x_k\|$ for every $\varepsilon>0$ and every $k$. 
	
	From $\|y_k\|\leq \|x_k\|$, we conclude that $\|f(\cdot)\|\leq \|f_0(\cdot)\|$. This implies that $f\in E(X)$ and $\|f\|\leq 1$.
	
	Observe that $f\in S(B_{E(X)}, f^\ast,\alpha)$. Indeed,
	\begin{align*}
	f^\ast(f)&=f^\ast\Big(\sum_{k=1}^m y_k\chi_{A_k}\Big)=\sum_{k=1}^m x^\ast_{k}(y_k)> \sum_{k=1}^m x^\ast_{k}(x_k)-\sum_{k=1}^m \alpha_k\\
	&=f^\ast(f)-\sum_{k=1}^m \alpha_k\geq 1-\alpha.
	\end{align*}
	
	Note that now $1-\alpha \leq \|f\|\leq 1.$ In the following we may assume that $f$ is of norm 1, if necessary, we use $f/\|f\|$ instead.

	Now we will show, that $f$ is indeed a $\Delta$-point, i.e.

	there exist $f_i\in B_{E(X)}$ such that $\|f - 1/n \sum_{i=1}^n f_i\|<\delta$ and for every $i\in\{1,\dots,n\}$ we have $\|f-f_i\|\geq 2-\varepsilon$.

	Define $f_i = \sum_{k=1}^m y_{ki} \chi_{A_k}.$
	From $\|y_{ki}\|\leq \|x_k\|$, we conclude that $\|f_i(\cdot)\|\leq \|f_0(\cdot)\|$. This implies that $f_i\in E(X)$ and $\|f_i\|\leq 1$ for every $i\in\{1,\dots,n\}$.

	Secondly, it is not hard to notice that if $s\in A_k$ then
	\begin{align*}
	    \Big\Vert f(s) - \frac{1}{n} \sum_{i=1}^n f_i(s)\Big \Vert &=
	    \Big\Vert y_k - \frac{1}{n}\sum_{i=1}^n y_{ki}\Big\Vert < \delta \|x_k\|, 
	\end{align*}
	and therefore $\|f - 1/n \sum_{i=1}^n f_i\| < \delta \|f_0\|\leq\delta.$
	
	Finally, if $s\in A_k$, then
	\begin{align*}
	    \|f(s) - f_i(s)\| &
	    = \|y_k - y_{ki}\| \geq (2-\varepsilon)\|x_k\|,
	\end{align*}
	and therefore, 
	\[
	\|f-f_i\|>(2-\varepsilon)\|f_0\|=2-\varepsilon,
	\]
	which completes the proof.

\end{proof}

From Theorem~\ref{thm: Köthe-Bochner} we immediately have the following.

\begin{cor}\label{cor: DD2P in L_p}
  Let $X$ be a Banach space, $(\Omega, \mathcal{A}, \mu)$ be a finite measure space, and $1\leq p\leq\infty$.
 If $X$ has the DD2P (resp. DLD2P, convex DLD2P), then $L_p(\mu,X)$ also has the DD2P (resp. DLD2P, convex DLD2P) whenever $L_p(\mu)\neq \{0\}$.
\end{cor}

\begin{rem}\label{rem: L_1 and L_infty Daugavet}
  If $\mu$ does not have any atoms, then both $L_1(\mu,X)$ and $L_\infty(\mu,X)$ even have the Daugavet property for any Banach space $X$ (see \cite[p.~81]{Werner2001}). 
\end{rem}

\section{Projective tensor product}\label{sec: proj ten}

We start this section by recalling the very recent definition of the weak operator Daugavet property, which is (formally) a strengthening of the Daugavet property (see \cite[Remark~5.3]{Martin2020}). 

\begin{defn}[{\cite[Definition~5.2]{Martin2020}}]\label{def: WODP}
  A Banach space $X$ has the \emph{weak operator Daugavet property} (WODP in short) if, given $x_1, \dots, x_n\in S_X$, $\varepsilon>0$, a slice $S$ of $B_X$, and $z\in B_X$, there is a $u\in S$ and $T\colon X\to X$ with $\|T\|\leq 1+\varepsilon$, $\|Tu-z\|\leq \varepsilon$, and $\|Tx_i-x_i\|\leq \varepsilon$ for every $i\in \{1,\dots,n\}$.
\end{defn}

Examples of spaces satisfying WODP include $L_1$-preduals with the Daugavet property and $L_1(\mu, X)$ whenever $\mu$ is atomless and $X$ is an arbitrary Banach space (see \cite{Rueda2019}). It is an open question whether $X\pten Y$ has the Daugavet property whenever $X$ and $Y$ both have the Daugavet property (see \cite[Section~6, Question~(3)]{Werner2001}). In \cite[Theorem~5.4]{Martin2020}, M.~ Mart\'{i}n and A.~Rueda Zoca proved that the WODP is stable by forming the projective tensor product, hence providing a large class of Banach spaces where the answer to the aforementioned question is positive. However, for the DLD2P of $X\pten Y$ it suffices to assume that $X$ has the WODP (cf. \cite[Proposition~5.7, Remark~5.9]{Rueda2019}). 

\begin{prop}\label{prop: DLD2P of proj ten}
  Let $X$ and $Y$ be Banach spaces. If $X$ has the WODP, then $X\pten Y$ has the DLD2P.
  \end{prop}

  It is known that for the LD2P of $X\pten Y$ it suffices to assume that $X$ has the LD2P (see \cite[Theorem~2.7]{Abrahamsen2013}). Yet we do not know whether in Proposition~\ref{prop: DLD2P of proj ten} it suffices to assume that $X$ has the DLD2P. Fortunately, it turns out that the convex DLD2P behaves similarly to the LD2P in this manner.

\begin{prop}\label{prop: convex DLD2P and proj tensor product}
  Let $X$ and $Y$ be Banach spaces. If $X$ has the convex DLD2P, then $X\pten Y$ has the convex DLD2P too.
\end{prop}
\begin{proof}
Consider a slice $S(B_{X\pten Y}, A, \alpha)$, where $A$ is a norm one operator from $Y$ to $X^\ast$ and $\alpha>0$. Let $\varepsilon>0$. We need to show that there are $z, z_1,\dots, z_n\in B_{X\pten Y}$ such that $z\in S(B_{X\pten Y}, A, \alpha)\cap \Delta_{X\pten Y}$, $\|z-1/n\sum_{i=1}^n z_i\|\leq \varepsilon$, and $\|z-z_i\|\geq 2-\varepsilon$ for every $i\in \{1,\dots,n\}$. 

Find $y\in S_Y$ such that $(1-\alpha/2)\|Ay\|\geq 1-\alpha.$ Take $x^\ast=Ay/\|Ay\|$ and consider the slice $S(B_X, x^\ast, \alpha/2)$. Since $X$ has the convex DLD2P, then there are $x, x_1,\dots, x_n\in B_{X}$ such that $x\in S(B_X, x^\ast, \alpha/2)\cap \Delta_{X}$, $\|x-1/n\sum_{i=1}^n x_i\|\leq \varepsilon$, and $\|x-x_i\|\geq 2-\varepsilon$ for every $i\in \{1,\dots,n\}$.

Denote by $z=x\otimes y$ and $z_i=x_i\otimes y$ for for every $i\in \{1,\dots,n\}$. Then, clearly $z,z_i\in B_{X\pten Y}$, and $z\in S(B_{X\pten Y}, A, \alpha)$, because
\[
A(z)=(Ay)x=\|Ay\|x^\ast(x)\geq \|Ay\|(1-\alpha/2)\geq 1-\alpha.
\]
Moreover, $\|z-z_i\|=\|x-x_i\|\geq 2-\varepsilon$ and
\[
\Big\Vert z-1/n\sum_{i=1}^n z_i\Big\Vert=\Big\Vert(x-1/n\sum_{i=1}^n x_i)\otimes y\Big\Vert=\Big\Vert x-1/n \sum_{i=1}^n x_i\Big\Vert\leq \varepsilon.
\]
From the arbitrariness of $ \varepsilon$ we conclude that  $X\pten Y$ has the convex DLD2P.
 
\end{proof}

\begin{rem}
  Notice that it follows easily from the proof of Proposition \ref{prop: convex DLD2P and proj tensor product} that if $x\in S_X$ is a $\Delta$-point in $X$ then $x\otimes y$ is a $\Delta$-point in $X\pten Y$ for every $y\in S_Y.$
\end{rem}
In order to  apply Proposition~\ref{prop: convex DLD2P and proj tensor product} in vector-valued Lipschitz-free spaces, we need to introduce some notation.

Given a metric space $M$ with a designated origin $0$ and a Banach space $X$, we will denote by $\Lip(M,X)$ the Banach space of all $X$-valued Lipschitz functions on $M$ which vanish at $0$ under the standard Lipschitz norm
$$\Vert f\Vert:=\sup\left\{ \frac{\Vert f(x)-f(y)\Vert}{d(x,y)}\ :\ x,y\in M, x\neq y \right\} .$$
Note that, $\Lip(M,X^*)$ is itself a dual Banach space. In fact, the map
$$\begin{array}{ccc}
\delta_{m,x}:\Lip(M,X^*) & \longrightarrow & \mathbb R\\
f & \longmapsto & f(m)(x)
\end{array}$$
defines a linear and bounded map for each $m\in M$ and $x\in X$. In other words, $\delta_{m,x}\in \Lip(M,X^*)^*$. If we define
$$\Free(M,X):=\overline{\spa}\{\delta_{m,x}\ :\ m\in M, x\in X\},$$
then we have that $\mathcal F(M,X)^*=\Lip(M,X^*)$. We write $\mathcal F(M):=\mathcal F(M,\mathbb R)$. It is known that $\mathcal F(M,X)=\mathcal F(M)\pten X$ (see \cite[Proposition 2.1]{Becerra2018}). 

\begin{cor}\label{cor: convexDLD2P in F(M)}
  Let $M$ be a complete metric space. Then the following assertions are equivalent:
  \begin{itemize}
    \item[(i)] $\mathcal{F}(M)$ does not have any strongly exposed point;
    \item[(ii)] $\mathcal{F}(M)$ has the convex DLD2P;
    \item[(iii)] $\mathcal{F}(M,X)$ has the convex DLD2P for every nontrivial Banach space $X$.
  \end{itemize}
\end{cor}

\begin{proof}
  (i) $\Leftrightarrow$ (ii) follows from \cite[Theorem~1.5]{Aviles2019}.

  (ii) $\Rightarrow$ (iii) Let $X$ be a Banach space. Since $\mathcal{F}(M,X)=\mathcal{F}(M)\pten X $ and  $\mathcal{F}(M)$ has the convex DLD2P, then $\mathcal{F}(M,X)$ has the convex DLD2P by Proposition~\ref{prop: convex DLD2P and proj tensor product}.

  (iii) $\Rightarrow$ (ii) is clear.
\end{proof}

\begin{rem}\label{rem: convex DLD2P in pten}
  Given a metric space $M$ with a designated origin $0$, the existence of a Banach space $X$ such that $\mathcal F(M,X)$ has the convex DLD2P does not imply that $\mathcal F(M)$ has the convex DLD2P. Indeed, given $X=L_1([0,1])$, it follows that
  $$\mathcal F(M,X)=\mathcal F(M)\pten X=\mathcal F(M)\pten L_1([0,1])=L_1([0,1],\mathcal F(M))$$
  has the Daugavet property (see Remark~\ref{rem: L_1 and L_infty Daugavet}), hence also the convex DLD2P, for every metric space $M$.
  \end{rem}

\section*{Acknowledgments}
The authors would like to thank the anonymous referee for a careful
reading of the manuscript and for suggestions that improved the exposition, in particular, for pointing out the recent paper \cite{Hardtke2020}.


\end{document}